\documentclass[reqno,12pt]{amsart}
\usepackage{amsfonts}
\usepackage[dvips]{color}

\newtheorem{theorem}{Theorem}[section]
\newtheorem{proposition}[theorem]{Proposition}
\newtheorem{lemma}[theorem]{Lemma}
\newtheorem{corollary}[theorem]{Corollary}

\newtheorem{remark}{Remark}[section]

\newcommand{\ra}{\rightarrow}

\makeatletter
\@addtoreset{equation}{section}
\makeatother

\newcommand {\IZ}{\mathbb{Z}}                        

\renewcommand\thefigure{\thesection.\@arabic\c@figure}
\renewcommand\thetable{\thesection.\@arabic\c@table}

\bibdata{prob}
\bibstyle{alpha}

\title[Exponential convergence for the FA1f model]{Exponential convergence for the Fredrikson-Andersen one spin facilitated model}

\author{T. Mountford, G. Valle}
\thanks{G. Valle was supported by CNPq grant 305805/2015-0 and Universal CNPq project 482519/2012-6. Both authors were supported by CNPq Science without Borders grant 402215/2012-5.}

\address{
\newline
Thomas Mountford
\newline
D\'epartement de Math\'ematiques, \'Ecole Polytechnique F\'ed\'erale,
\newline 1015 Lausanne, Switzerland.
\newline
e-mail: {\rm \texttt{thomas.mountford@epfl.ch}}
\newline
\newline
Glauco Valle
\newline
UFRJ - Departamento de m\'etodos estat\'{\i}sticos do Instituto de Matem\'atica,
\newline  Caixa Postal 68530, 21945-970, Rio de Janeiro, Brasil.
\newline
e-mail: {\rm \texttt{glauco.valle@im.ufrj.br}}
}

\begin{document}

\begin{abstract}
We prove exponential convergence to equilibrium for the Fredrikson-Andersen one spin facilitated model on bounded degree graphs satisfying a subexponential, but larger than polynomial, growth condition. This was a classical conjecture related to non-attractive spin systems. Our proof rely on coupling techniques based on Harris graphical construction for interacting particle systems.
\end{abstract}

\maketitle

\section{Introduction} \label{intro}

Let $G=(V,E)$ be a countable connected graph of bounded degree $\kappa \ge 1$ and let $d: V \times V \rightarrow \mathbb{Z}_+$ be the usual graph distance with respect to $G$. We also denote $x \sim y$, if $x,y \in V$ are nearest neighbor sites, i.e. $d(x,y)=1$. We consider here the Fredrikson-Andersen one spin facilitated model (FA1f) on $G$ which is a continuous time spin system $\eta = (\eta_t)_{t \ge 0}$ with state space $\Omega = \{0,1\}^{V} - \{\bar{0}\}$, where $\bar{0}$ is the identically zero configuration, and transition rates $c(\eta,\tilde{\eta})$ equal to zero except for 
$$
c(\eta,\eta^x) = 
\left\{
\begin{array}{cl}
\lambda \, &, \textrm{ if } \eta^x(x) = 1 \textrm{ and } \sum_{y \sim x} \eta(y) > 0, \\
\mu \, &, \textrm{ if } \eta^x(x) = 0 \textrm{ and } \sum_{y \sim x} \eta(y) > 0,  
\end{array}
\right. 
$$
for some $\lambda, \mu > 0$, where $\eta^x$ is the configuration obtained from $\eta$ by flipping the spin at site $x$. We will suppose $\lambda + \mu = 1$, which can be obtained in a standard way by a time rescaling. Then we can fix $q = \lambda = 1-\mu \, \in (0,1)$ as the unique parameter of the process whose evolution can be informaly described as follows: Each site waits an exponential time of parameter one, independently of any other site, and by this time, if at least one of its neighbors have value one, it takes the value $1$ with probability $q$ and the value $0$ with probability $1-q$. From now on $q \in (0,1)$ is to be considered fixed. 

As usual in interacting particle systems theory, if $\eta_t(x) = 1$ we will say that site $x$ is occupied by a particle at time $t$ (or simply, that $x$ is an occupied site at time $t$). Otherwise, we say that site $x$ is empty.

The Bernoulli product measure of parameter $q$, denoted here by $\nu_q$, is invariant, in fact reversible, for the FA1f process $(\eta_t)_{t \ge 0}$. Another important feature of the FA1f process is that it is not attractive. 

We say that the graph $G$ satisfies a polynomial growth condition if for every $r>0$ and $x \in V$, the cardinality of the ball of radius $r$ around $x$ is bounded above by $\beta r^d$ for some $\beta >0$ and $d\ge 1$ not depending on $r$ and $x$. Under such condition, in \cite{b-all} it is studied the speed of convergence to equilibrium. It is shown (Theorem 2.1 in \cite{b-all}) that for $q>1/2$ and initial configurations with sufficiently large and spatially well distributed number of particles, then convergence of the finite dimensional distributions occurs exponentially fast in time with exponent of order $\big(t/\log(t)\big)^{1/d}$. 

Our aim is to improve the result in \cite{b-all} for $q$ sufficiently close to one by showing an exponential decay to equilibrium with an exponent of order $t$. Indeed we can consider a larger than polynomial growth condition but still subexponential. The graph $G$ satisfies a $(\vartheta,\tilde{\vartheta},\varepsilon)$-growth condition, for $\vartheta >0$, $\tilde{\vartheta}> 0$ and $\varepsilon \in (0,1)$, if the cardinality of the ball of radius $r$ around $x$ is bounded above by $\vartheta e^{\tilde{\vartheta} \, r^{1-\varepsilon}}$. So our main result is the following:

\medskip

\begin{theorem} \label{theorem:main} 
Let $G = (V,E)$ be a countable connected graph of bounded degree satisfying the $(\vartheta,\tilde{\vartheta},\varepsilon)$-growth condition. For $q$ sufficiently close to one, any given site $y \in V$ and every finite dimensional set $\Gamma \subset \Omega$, there exist constants $c = c (q,\vartheta,\tilde{\vartheta},\varepsilon)>0$ and $C=C(q,\vartheta,\tilde{\vartheta},\varepsilon,y,\Gamma)>0$ such that
$$
\big| \mathrm{P}^{\delta_y} \big( \eta_t \in \Gamma) - \nu_q (\Gamma) \big| \le C e^{-ct} \, ,
$$
where $\delta_y$ is the configuration with a single particle on site $y \in V$.
\end{theorem}

\medskip

\begin{remark}
The d-regular trees do not satisfy a $(\vartheta,\tilde{\vartheta},\varepsilon)$-growth condition and we do not think that our proof for Theorem \ref{theorem:main} can be adapted to this case. Related to this is important to point out that the growth condition is only used at one point in our proof, specifically at the proof of Lemma \ref{lemcount2}.
\end{remark}

\medskip

\begin{remark} As will become clear from the proof of Theorem \ref{theorem:main}, we get that the statement of Theorem \ref{theorem:main} still holds if we replace $\delta_y$ by a initial distribution $\nu$ such that for some $z \in V$ and
$m,M>0$ we have
$$
\nu \big( \min\{ d(x,z):\eta_0(x)=1 \} \ge r \Big) \le M e^{-m r}
$$ 
for every $r>0$. 
\end{remark}

\medskip

Let us start by describing the main steps in the proof of Theorem \ref{theorem:main} and how they lead to the verification of the statement. Fix $y \in V$ and $\Gamma$ a finite dimensional subset of $\Omega$, we will also identify it to a subset $B = B_\Gamma \subset V$ such that $\Gamma$ only depends on the configuration on sites of $B$. The main idea of the proof is to show that we can couple FA1f processes starting at $\delta_y$ and $\nu_q$ such that, outside an event with probability of order $e^{-ct}$, the FA1f process starting at $\delta_y$ restricted to sites in $B$ has the same configuration at time $t$ as the process starting with distribution $\nu_q$. 

It is clear that we only need to prove Theorem \ref{theorem:main} for $B=\{x\}$ for some $x\in V$. We suppose this from now on.

The coupling mentioned above is based on the Harris graphical construction of the FA1f process and
an associated percolation structure in dual time that allows us to identify when a fixed site $x \in V$ has the same configuration for both processes at a given time $t$. Let us start by describing the Harris graphical construction: Let $\big(\mathcal{P}_x\big)_{x \in V}$ be a family of rate one Poisson point processes on the half-line $(0,\infty)$ and $(\gamma_{x,n})_{x\in V, \, n\ge 1}$ be a family of iid Bernoulli random variables of parameter $q$ which is independent of the Poisson point processes. Then there exists a version of the FA1f process on the same probability space of $\big(\mathcal{P}_x,(\gamma_{x,n})_{n\ge 1}\big)_{x \in V}$ which is defined by
$$
\eta_t (x) = 
\left\{
\begin{array}{cl}
\gamma_{x,n} &, \ \sum_{y \sim x} \eta_{t-} (y) \ge 1 \ \mathrm{and} \ t \in \big[ T_{x,n} , T_{x,n+1} \big) \, , \ n\ge 1 \, ; \\
\eta_{t-}(x) &, \ \mathrm{ otherwise} \, .
\end{array}
\right.
$$
where the $(T_{x,n})_{n\ge 1}$ are the time marks in the Poisson point process $\mathcal{P}_x$, which will also be called \emph{decision times}. For each $x \in V$, we can decompose $\mathcal{P}_x$ in two independent Poisson point processes, one with parameter $q$ associated to points with marks $\gamma_{x,n}=1$, say $\mathcal{P}^{'}_x$, and its complement $\mathcal{P}^{''}_x$. Points in $\mathcal{P}^{'}_x$ will be called \emph{type-1 decision times} and points in $\mathcal{P}^{''}_x$ \emph{type-0 decision times}. We also call 
$$
\big( \big(\mathcal{P}_x\big)_{x \in V},(\gamma_{x,n})_{x\in V, \, n\ge 1}\big)
$$ 
the \emph{Harris scheme} for the FA1f model.

Using the above defition we obtain a pair of FA1f processes $(\eta_t,\tilde{\eta}_t)_{t\ge 0}$ starting from any bivariate initial distribution on $\Omega^2$ where both marginals evolve using the same Harris scheme. We are particularly interested in the case where the first marginal starts at $\delta_y$, for some $y \in V$, and the second one starts from the equilibrium measure $\nu_q$. In this case, we represent the probability associated to the process $(\eta_t,\tilde{\eta}_t)_{t\ge 0}$ by $P^{\delta_y,\nu_q}$.

We call a site $x \in V$ $t$-\emph{activated} if $\eta_t(x) = \tilde{\eta}_t(x)$. Our aim is to show that given $x \in V$, then outside an event of exponentially small probability with respect to $t$, $x$ is $t$-activated. 

Therefore Theorem \ref{theorem:main} follows from:

\medskip

\begin{proposition}
\label{prop:main}
Let $G = (V,E)$ be a countable connected graph of bounded degree satisfying the $(\vartheta,\tilde{\vartheta},\varepsilon)$-growth condition.
For $q$ sufficiently close to one and every $x, \, y \in V$, there exist constants $c=c(q,\vartheta,\tilde{\vartheta},\varepsilon)>0$ and $C=C(q,\vartheta,\tilde{\vartheta},\varepsilon,y,x)>0$ such that
$$
\mathrm{P}^{\delta_y,\nu_q} \big( \, x \textrm{ is not } t\textrm{-activated} \, \big) \le C e^{-ct} \, ,
$$
for every $t > 0$.
\end{proposition}

\medskip

To prove Proposition \ref{prop:main} we need a proper condition to guarantee that a given site $x$ is $t$-activated. The main idea is that $x$ is $t$-activated if it had the opportunity to choose its spin configuration before time $t$ (at the last possible allowed time) simultaneously for both processes. To use this idea we need to introduce some definitions and notation.
We can define the concept of dual path associated to a given pair $(x,t) \in V \times [0,+\infty)$ on a given time interval $[0,\tau]$ for some $\tau \in (0,t]$, which we call here a $\tau$-dual path. A $\tau$-dual path of $(x,t)$ is built on a realization of the FA1f process as a reversed time piecewise constant rightcontinuous path starting at $x$ such that changes are only possible at 
decision times (the choice between right and left continuous is not important for us, so we choose right continuity). 

Formally we have $(X(s))_{0 < s \le t-\tau}$ that starts at time $0$ at position $x$. It can be constant equal to $x$ or follow the realization of the process backwards in time until a certain decision time $t_1 \in \mathcal{P}_x \cap (\tau,t)$ (if no such point exists then the only possible path is the one that is constant equal to $x$), then at time $s_1 = t - t_1$ the step function jumps from $x$ to position $x_1$ chosen among one of its neighbors. By a finite induction procedure, if we have already had $k$ jumps and $X$ is at site $x_k$ at time $s= s_k \le t-\tau$ then, either $X(s) = x_k$ for $s \in (s_k,t-\tau]$, or, if $\mathcal{P}_{x_k} \cap (\tau,t-s_k) \neq \emptyset$, we can take $t_{k+1} \in \mathcal{P}_{x_k} \cap (\tau,t-s_k)$ such that the step function jumps at time $s_{k+1} = t - t_{k+1}$ to a site $x_{k+1}$ chosen among the neighbors of $x_k$. We denote the random set of all $\tau$-dual paths of $(x,t)$ by $\mathcal{D}(x,t,\tau)$.

For $\tau \in (0,t)$, a path in $X \in \mathcal{D}(x,t,\tau)$ is called an \emph{activated path}, if for some time $s \in (0,t-\tau]$, we have  $\eta_{t-s} \big( X(s) \big) = \tilde{\eta}_{t-s} \big( X(s) \big) = 1$. We denote by $\mathcal{A}(x,t,\tau)$ the random collection of all activated paths in $\mathcal{D}(x,t,\tau)$. 

\begin{lemma} \label{lem:ativo}
If $\mathcal{D}(x,t,\tau) = \mathcal{A}(x,t,\tau)$ for some $\tau \in (0,t)$, then $x$ is $t$-activated.  
\end{lemma}

\begin{proof}
The proof follows from a contradiction argument. Suppose that $x$ is not $t$-activated, we show that there exists a path $X \in \mathcal{D}(x,t,\tau)$ which is not activated. We construct $X$ by a finite number of steps as follows: 

\noindent \emph{Step 1:} Since $\eta_t (x) \neq \tilde{\eta}_{t} (x)$, then two cases may occur: 
\begin{itemize} 
\item[(i)] $\eta_{t-s} (x) \neq \tilde{\eta}_{t-s} (x)$ for all $s \in (0,t-\tau)$. In this case the constant path $X \equiv x$ is not $t$-activated and we stop at step 1.
\item[(ii)] $\eta_{t-s} (x) = \tilde{\eta}_{t-s} (x)$ for some $s \in (0,t-\tau)$. Thus $\mathcal{P}_x \cap (\tau,t) \neq \emptyset$ and we can take $t_1 = \sup \big\{ r \in \mathcal{P}_x \cap (\tau,t) : \eta_r (x) = \tilde{\eta}_{r} (x) \big\}$. At time $t_1$, either $\eta_{t_1} (y) = 0$
for all $y \sim x$ and $\tilde{\eta}_{t_1} (y) = 1$ for some $y \sim x$, or the same happens exchanging the roles of $\eta$ and $\tilde{\eta}$. Thus
there exists a neighbor $x_1$ of $x$ such that $\eta_{t_1} (x) \neq \tilde{\eta}_{t_1} (x)$. In this case, we consider that $X$ jumps to $x_1$ at time $s_1 = t-t_1$. 
\end{itemize}

Now by finite induction, we suppose that after Step $k$, for some $k \ge 1$, we have built our path $X$ up to time $s_k \le t-\tau$ such that $\eta_s \big( X(s) \big) \neq \tilde{\eta}_{s} \big( X(s) \big)$ for all $s \in (0,s_k)$. Suppose that $X(s_k) = x_k$ then we perform step $k+1$.

\smallskip

\noindent \emph{Step k+1:} Two cases may occur: 
\begin{itemize} 
\item[(i)] $\eta_{t-s} (x_k) \neq \tilde{\eta}_{t-s} (x_k)$ for all $s \in (s_k,t-\tau)$ and we put $X(s) = x_k$ for $s \in (s_k,t-\tau]$. Then $X$ is not activated and we stop at step $k+1$.
\item[(ii)] $\eta_{t-s} (x_k) = \tilde{\eta}_{t-s} (x)$ for some $s \in (s_k,t-\tau)$. Thus $\mathcal{P}_x \cap (\tau,t-s_k) \neq \emptyset$ and we can take $t_{k+1} = \sup \big\{ r \in \mathcal{P}_{x_k} \cap (\tau,t-s_k) : \eta_r (x) = \tilde{\eta}_{r} (x) \big\}$. At time $t_{k+1}$, there exists a neighbor $x_{k+1}$ of $x$ such that $\eta_{t_{k+1}} (x_{k+1}) \neq \tilde{\eta}_{t_{k+1}} (x_{k+1})$. In this case, we consider that $X$ jumps to $x_{k+1}$ at time $s_{k+1} = t-t_{k+1}$. 
\end{itemize}

The number of steps is clearly stochastically dominated by a Poisson distribution of parameter one and then it is finite almost surely.
\end{proof}

\medskip

From Lemma \ref{lem:ativo}, we have that Proposition \ref{prop:main} follows from the next result.

\begin{proposition}
\label{prop:main2}
Let $G = (V,E)$ be a countable connected graph of bounded degree satisfying the $(\vartheta,\tilde{\vartheta},\varepsilon)$-growth condition. For $q$ sufficiently close to one, $\sigma <1/4$ sufficiently small, and every $x, \, y \in V$, there exist constants $c=c(q,\vartheta,\tilde{\vartheta},\varepsilon,\sigma)>0$ and $C=C(q,\vartheta,\tilde{\vartheta},\varepsilon,\sigma,x,y)>0$ such that
$$
\mathrm{P}^{\delta_y,\nu_q} \big( \, \mathcal{D}(x,t,(1-\sigma) t) \neq \mathcal{A}(x,t,(1-\sigma) t) \, \big) \le C e^{-ct} \, ,
$$
for every $t>0$.
\end{proposition}

\medskip

\section{Proof of Proposition \ref{prop:main2}}

Since $V$ is infinite and $G$ has bounded degree, $G$ contains a copy of $\mathbb{Z}_+$, i.e., there exists $\mathcal{Z} = \{z_i\}_{i\in \mathbb{Z}_+} \subset V$ such that $d(z_i,z_{i+1}) = 1$ for every $i \in \mathbb{Z}_+$. We denote by $\mathcal{G}$ the subgraph $(\mathcal{Z},\mathcal{E}) \subset G$, where $\mathcal{E}$ is the collection of edges $\{z_i,z_{i+1}\}$, $i \in \mathbb{Z}_+$. 

\smallskip

The proof of Proposition \ref{prop:main2} is made of three major stages. The first stage is a warming up argument for the process which allows us to guarantee that, outside an event of exponentially small probability, we have an appropriately concentrated and sufficiently large number of occupied sites at time $t/4$ on $\mathcal{Z}$. The second stage is based on the construction of a percolation struture that will be used in the third stage to show that, also outside an event of exponentially small probability, all dual paths in $\mathcal{D}(x,t,(1-\sigma) t)$ touchs another path that is capable of transporting ones from time $t/4$. Finally we use the results obtained in the three stages to prove that if the conditions described above for the first and third stages are met then all paths in $\mathcal{D}(x,t,(1-\sigma) t)$ are $t$-activated. The idea is to show that all paths in $\mathcal{D}(x,t,(1-\sigma) t)$ touch some space time point in $V \times ((1-\sigma) t,t)$ where $\eta$ and $\tilde{\eta}$ are equal to one and then we need a warming up argument to populate the graph structure for both processes with a sufficiently large number of occupied sites at time $t/4$ (first stage), a suitable percolation structure to define paths that are capable of transporting ones from time $t/4$ to time interval $[(1-\sigma) t,t]$ (second stage) and a final step to show that we can connect all dual paths in $\mathcal{D}(x,t,(1-\sigma) t)$ to these tranporting paths. After we have established the three stages described above, we finish the section with the proof of Proposition \ref{prop:main2}.

\begin{remark}
Our proof requires $q$ to be sufficiently close to one. In the second stage, $q$ are going to be replaced as a function of a renormalization parameter $K$ which should be taken sufficiently large. Then some of the results in Sections \ref{secondstage} and \ref{thirdstage} are stated in terms of $K$ instead of $q$.
\end{remark}

\medskip \bigskip

\subsection{First Stage.} \label{firststage}

We now describe the first stage in the proof of Proposition \ref{prop:main2}. Fix $0\le \tau < \tilde{\tau}$, We say that a path $Y: [\tau,\tilde{\tau}] \rightarrow \Omega$ is a $(\tau,\tilde{\tau})$-navigated path, or simply a navigated path, for the FA1f process $(\eta_t)_{t\ge 0}$ if 
\begin{enumerate}
\item[(i)] Y is a c.a.d.l.a.g step function;
\item[(ii)] $d(Y_s,Y_{s-}) \le 1$, $s \in [\tau,\tilde{\tau}]$;
\item[(iii)] $(\eta_t)_{t\ge 0}$ if $\eta_s \big( Y(s) \big) = 1$ for all $\tau \le s \le \tilde{\tau}$. 
\end{enumerate} 
We are interested in the events $\mathcal{N}((x_0,x_1,...,x_n),s,t)$ which, for $x_0$, $x_1$, ..., $x_n \in V$, $s \ge 0$ and $t > s$, is defined as the event that there exists a $(\tau,\tilde{\tau})$-navigated path for some $s \le \tau < \tilde{\tau} \le t$ that starts at site $x_0$ and visits all sites $x_1$,...$x_n$. 

\medskip

We first show how to construct a navigated path $Y$ from a site $x \in V$, occupied at time $\tau$, to a site $\tilde{x} \in V$. So the process starts at $Y(\tau) = x$. Given $Y_t$, for some $t>\tau$, the process remains at its current site until the first decision time $t^{'} > t$ in $\mathcal{P}_y$ where either $y=Y_t$ or $y$ is one of its neighbors that are closer to $\tilde{x}$ in graph distance, i.e. $d(\tilde{x},y) = d(\tilde{x},Y_t)-1$. If $y=Y_t$ and the spin at site $Y_t$ remains $1$ at time $t^{'}$ there is no change of position, otherwise $Y$ jumps to an occupied neighbor among the closest to $\tilde{x}$. If $y \sim Y_t$ then $Y$ jumps to the neighboring site if it takes value $1$. When the process arrives at site $\tilde{x}$ it remains there and do not jump anymore.  

To each navigated path $Y$ to a site $\tilde{x}$ starting at site $x$ by time $\tau$ we can define the process $S_t = d(Y_t,\tilde{x})$, $t>\tau$, which is a continuous time nearest-neighbor random walk on $\mathbb{Z}_+$ having $0$ as an absorbing state that decreases by one at rate greater or equal to $q$ and increases by one at rate smaller or equal to $1-q$. 
The expected time for the navigated path to arrive at $\tilde{x}$, 
$T= \inf \{ s > 0 : Y_{\tau + s} = \tilde{x} \}$, is bounded above by the expected time to arrive at $d(x, \tilde{x})$ for a simple nearest neighbor random walk that jumps to the right with probability $q$ and starts at zero. Thus
\begin{equation}
\label{fastnavigated}
E[T] \le \frac{d(x, \tilde{x})}{2q-1} \, .
\end{equation}


\medskip

\begin{lemma} \label{lem:navigated}
Let $q > 1/2$ and $\nu$ be a initial distribution for the FA1f process satisfying that the distribution of $\min\{d(x,z_0):\eta_0(x)=1\}$ has exponentially decaying tail. Therefore for every $L < 2q-1$, there exists $c=c(q,L,\nu)>0$, $C=C(q,L,\nu)>0$ depending on $q$, $L$ and $\nu$ such that
$$
\mathrm{P}^{\nu} \big( \mathcal{N}((z_0,z_1,...,z_{Lt}) ,0,t) \big) \ge 1 - C e^{-ct} \, ,
$$
for every $t>0$. Futhermore, if $\nu = \delta_y$ then we can choose $c$ depending only on $q$ and $L$.
\end{lemma}

\smallskip

\begin{remark}
We can take $\nu \in \{\nu_q, \ \delta_y, \ y\in V\}$ in the statement of Lemma \ref{lem:navigated}. Clearly $\delta_y$, for a fixed $y\in V$ satisfies the condition in the statement. For $\nu_q$, the random variable $\min\{d(x,z_0):\eta_0(x)=1\}$ is stochastically dominated by a geometric distribution with parameter $q$, which also implies the condition in the statement.
\end{remark}

\begin{proof}
Let $(\eta_t)_{t \ge 0}$ be a FA1f process starting at $\nu$. Take $y$ a random site in $V$ satisfying that $d(y,z_0) = W = \min\{d(x,z_0):\eta_0(x)=1\}$. It is clear that
$$
\mathrm{P}^{\nu} \big( \mathcal{N}((z_0,z_1,...,z_{Lt}) ,0,t)^c \big)
$$
is bounded above by
\begin{equation}
\label{eq:navpathcond}
C e^{-ct} + \sum_{j=0}^{\lfloor \theta t \rfloor} 
\mathrm{P}^{\nu} \big( \mathcal{N}((z_0,z_1,...,z_{Lt}) ,0,t)^c \big| W=j \big) \mathrm{P}(W=j) \, ,
\end{equation}
for any constant $\theta$. We fix $\theta = \frac{(2q-1)-L}{2}$ and $L' = L+\theta$ which is smaller than $2q-1$. Therefore, we need to show that 
$$
\mathrm{P}^{\nu} \big( \mathcal{N}((z_0,z_1,...,z_{Lt}) ,0,t)^c \big| W=j \big)
$$
decays exponentially fast as $t \ra \infty$ uniformly for $j \in \{0,1,2,..., \lfloor \theta t \rfloor \}$.

Now fix $j$ as above and an occupied site at time $0$, $y_0 \in V$, such that $d(y,z_0)=j$. 
Fix a nearest-neighbor path $y_0,...,y_{j-1},z_0$.
We have that $\mathcal{N}((y_0,...,y_{j-1},z_0,...,z_{Lt}) ,0,t)$ happens if we build a concatenation of $Lt+j$ navigation paths between the pairs $(y_{0},y_1)$, $(y_{j-1},z_0)$, $(Z_0,z_1)$, ... , $(z_{Lt-1},z_{Lt})$ such that the time length of the concatenated path is smaller than $t$. so we build these paths using the construction described above and denote their time length by $T_1$, ... , $T_{Lt+j}$. By the Strong Markov property, these are independent random variables whose distributions are stochasticaly dominated by the absorbing time at the origin of a homogeneous positive recurrent nearest neighboor continuous time random walk on $\mathbb{N}$ starting at one. Moreover, by (\ref{fastnavigated}) this absorbing time has expectation $(2q-1)^{-1}$. Basic properties of random walks allows us to show that the distribution of the times $T_j$ have finite moment generating function on some interval around zero, details are left to the reader. By Crámer Theorem, since $L < L' < 2q-1$, we have that
$$
\mathrm{P} \Big( \sum_{l=1}^{Lt+j} T_l \ge t \Big) \le \mathrm{P} \Big( \frac{1}{L' t} \sum_{l=1}^{L't} T_l \ge \frac{1}{L'} \Big) \le C e^{-ct} \, ,
$$
for some constant depending on $q$, $L$ and $\nu$.   

To finish the proof we have to consider the case $\nu = \delta_y$. In this case we do not have the first term in \eqref{eq:navpathcond}, so it is clear that we can choose $c$ not depending on $y$ by taking $C$ sufficiently large depending on it.
\end{proof}

\medskip

We finish the first stage by using Lemma \ref{lem:navigated} and a comparison with a discrete time contact process to control the number of occupied sites among $\{z_0,z_1,,...,z_{Lt}\}$ at time $t$. 

We will use the Harris scheme to couple $(\eta_t)_{t \ge 0}$ to a discrete time contact process $(\xi_n)_{n\ge 0}$ which is a discrete Markov Process with state space $\{0,1\}^\mathbb{Z}_+$ such that given $\xi_n$ we have that $(\xi_{n+1}(j))_{j \in \mathbb{Z}_+}$ are condicionally independent and, for some $p, \hat{p} \in (0,1)$, 
\begin{eqnarray} \label{eq:dc}
\lefteqn{ \!\!\!\!\! \mathrm{P} \big( \xi_{n+1} (j) = 1 \big| \xi_n \big) = } \nonumber \\
& \qquad \qquad \left\{
 \begin{array}{cl}
 p &, \textrm{ if } \xi_{n} (j) = 1 \, , \\
 1 - (1-\hat{p})^{\xi_n(1)} &, \textrm{ if } j = 0 , \ \xi_{n} (j) = 0 \, , \\
 1 - (1-\hat{p})^{\xi_n(j+1) + \xi_n(j-1)} &, \textrm{ if } j > 0 , \ \xi_{n} (j) = 0 \, .
 \end{array}
\right.
\end{eqnarray}

\medskip

\begin{lemma} \label{lem:coupling}
For $q$ sufficiently close to one and $\theta > 0$ sufficiently large, there exists $p = p(q,\theta)$ and $\hat{p} = \hat{p}(q,\theta)$ in $(0,1)$ and a coupling between the the FA1f process, $(\eta_t)_{t \ge 0}$, and a discrete contact process of parameters $p$ and $\hat{p}$, $(\xi_n)_{n\ge 0}$, such that if $\eta_{0} (z_j) \ge \xi_0 (j)$ for every $j \ge \mathbb{Z}_+$ then almost surely $\eta_{\theta n}(z_j) \ge \xi_n (j) = 1$ for every $j \ge \mathbb{Z}_+$. Futhermore, $p(q,\theta) \rightarrow 1$ and $\hat{p}(q,\theta) \rightarrow 1$ as $q \rightarrow 1$ and $\theta \rightarrow \infty$.
\end{lemma}

\begin{remark}
Although we lose information when we replace the FA1f process by the discrete contact process, which should be clear by the proof of Lemma \ref{lem:coupling}, we need it due to the lack of attractivity of the FA1f and the need to have some proper estimates on the density of ones by time $t$. Moreover, we can rely on the fact that the discrete contact process is well known, see from instance Durrett \cite{d1,d2}. On Section \ref{secondstage} we discuss another discrete (but dual) time contact process  and we recall some properties of such processes.
\end{remark}

\begin{proof}
We will consider a version of $(\xi_n)_{n\ge 0}$ built using the Harris scheme for the FA1f process. We consider a partition of the time interval into disjoint consecutive intervals of length $\theta$. So considering the values of $\eta_{\theta n}$ on $\mathcal{Z}$ and $\xi_n$ and supposing that $\eta_{\theta n} (z_j) \ge \xi_n (j)$ for every $j \ge 1$, we want to use the restriction of the Harris scheme to the time interval $(\theta n, \theta (n+1)]$ to specify $\xi_{n+1}$ such that we still have $\eta_{\theta (n+1)} (z_j) \ge \xi_{n+1} (j)$ for every $j \ge 1$. Once this specification is done the proof follows from induction. 

\smallskip

We have to obtain the parameters $p$ and $\hat{p}$ in the definition of the transition probabilities in (\ref{eq:dc}).
Put $\xi_0 = \eta_0$ and fix $j \ge 1$. To obtain $\xi_{n+1} (j)$ from $\xi_{n}$ using the Harris scheme define $W^{'}_k$ as the waiting time from $\theta n$ to the first occurence of a time in $\mathcal{P}^{'}_{z_k}$, i.e.
$$
W^{'}_k = \min \{ \mathcal{P}^{'}_{z_k} \cap [\theta n, \infty)  \} - \theta n \, ,
$$ 
and $W^{''}_k$ is defined analogously using $\mathcal{P}^{''}_{z_k}$. 

we only need to consider the three complementary cases below: 

\smallskip

\noindent \textbf{Case $\eta_{\theta n} (z_j) = \xi_n (j) = 1$:} 

Here if $\mathcal{P}^{''}_{z_j} \cap [\theta n, \theta (n+1)] \neq \emptyset$ then $\eta_{\theta (n+1)} (z_j)=1$. This happens with probability 
$$
p^{\prime} = \mathrm{P} \big( W^{''}_j > \theta \big) = e^{-\theta(1-q)} \, .
$$
Thus we simply fix $p = p^{'}$. 

\smallskip

\noindent \textbf{Case $\xi_n (j) = 0$ with $\xi_n (j \mp 1) = 0$ and $\eta_{\theta n} (z_{j \pm 1}) = \xi_n (j \pm 1) = 1$}:

Suppose $\xi_n (j - 1) = 0$ and $\eta_{\theta n} (z_{j + 1}) = \xi_n (j + 1) = 1$, the other case is analogous. If $\xi_{n+1} (j) = 1$ we should have $\eta_{\theta (n+1)} (z_j)=1$ which happens in the event 
$$
\{ W^{''}_{j} > \theta \} \cap \big\{ W^{'}_j < \big( \theta \wedge W^{''}_{j+1} \big) \big\} \, .
$$
By a standard computation, the probability of this previous event is equal to
$$
p^{''} = q e^{-\theta(1-q)} (1-e^{-\theta}) \,  .
$$
Then we should have $\hat{p} \ge p^{''}$.

\smallskip

\noindent \textbf{Case $\xi_n (j) = 0$ with $\eta_{\theta n} (z_{j-1}) = \xi_n (j-1) = \eta_{\theta n} (z_{j+1}) = \xi_n (j+1) = 1$:}
In this case, to guarantee that $\xi_{n+1} (j) = 1$ implies $\eta_{\theta (n+1)} (z_j)=1$ we use the event
$$
\{ W^{''}_{j} > \theta \} \cap \big\{ W^{'}_j < \big(\theta \wedge (W^{''}_{j-1} \vee W^{''}_{j+1})\big) \big\}  \, .
$$
Its probability can be computed explicitly as
$$
p^{\prime \prime \prime} = q e^{-2 \theta (1-q)} \left[ 2(1-e^{-\theta}) - \frac{1 - e^{-\theta (2-q)}}{(2-q)} \right] \, .
$$
We also should have $\hat{p} \ge 2 p^{'''} - (p^{'''})^2$.

\smallskip

From the second and third cases above, it is enough to take $\hat{p} = \max\{ p^{''}, 2 p^{'''} - (p^{'''})^2\}$. Finally it is clear from the definitions that $p(q,\theta) \rightarrow 1$ and $\hat{p}(q,\theta) \rightarrow 1$ as $q \rightarrow 1$ and $\theta \rightarrow \infty$.
\end{proof} 

\medskip

\begin{remark} We remark that the oriented percolation model from \cite{d2} is not exactly the one associated to the one-sided discrete contact process above, but the results remain valid with some straightforward adaptation of the arguments there. Indeed by a standard coupling argument we can show that the discrete contact process is stochastically above a pair of oriented percolation models evolving respectively on $\{(x,y) \in \mathcal{Z}_+ : x+y \textrm{ is even}\}$ and $\{ (x,y) \in \mathcal{Z}_+ : x+y \textrm{ is odd} \}$. See also \cite{d2,d3} and the discussion on discrete time contact processes on section \ref{secondstage} of this paper
\end{remark}

\bigskip

\begin{proposition}
\label{lemma:densidade}
Let $\nu$ be a initial distribution for the FA1f process satisfying that the distribution of $\min\{d(x,z_0):\eta_0(x)=1\}$ has exponentially decaying tail. For each $\rho \in (0,1)$, there exists $q_0$ such that for $q>q_0$ and $L < \frac{2q-1}{2}$ there exist $c>0$ and $C>0$ depending on $q$, $\nu$, $\rho$ and $L$ such that 
$$
\mathrm{P^\nu} \Big( \frac{\# \big\{ j \in \{0,1,...,Lt-1\}  \, : \ \eta_t (z_j) =1 \big\}}{Lt} \le \rho \Big) \le C e^{-ct} \, ,
$$ 
for every $t>0$. Futhermore, if $\nu = \delta_y$ then we can choose $c$ depending only on $q$, $\rho$ and $L$.
\end{proposition}

\begin{proof}
Apply Lemma \ref{lem:navigated} considering navigated paths on time interval $[0,t/2]$ and we have that 
$$
\mathrm{P}^{\nu} \big( \mathcal{N}\big( (z_0,z_1,...,z_{Lt-1}) ,0, t/2 \big)^c \big) \le \tilde{C} e^{-\tilde{c} t} \, ,
$$
for $\tilde{c}>0$ and $\tilde{C}>0$ depending on $q$ and $L$.
So we only need to show that given $\mathcal{N}\big( (z_0,z_1,...,z_{Lt}) ,0, t/2 \big)$, the probability of 
$$
\Big\{  \frac{\# \big\{ j \in \{0,1,...,Lt-1\}  \, : \ \eta_t (z_j) =1 \big\}}{Lt} \le \rho \Big\}
$$
decays exponentially fast if $q$ is sufficiently large. 

Now we are going to use the coupling with the discrete time contact process and a small renormalization argument. Let us fix $R >0$ that should be taken large. We make a partition of $\{0,1,...,Lt-1\}$ into the sets $\Gamma_l = \{(l-1)R, ..., lR - 1\}$, $1 \le l \le \big\lceil (Lt + 1)/R \big\rceil$. For $\alpha \in (0,1)$ let $W_l^\alpha$ be Bernoulli random variables defined as follows: $W^\alpha_l = 1$ if the number of occupied sites in $\Gamma_l$ by time $t$ is bounded below by $\alpha R$, otherwise $W_l^\alpha = 0$. 

Recall that we are conditioning on $\mathcal{N}\big( (z_0,z_1,...,z_{Lt}) ,0, t/2 \big)$ and each set $\Gamma_l$ has an occupied site during some time in the interval $[0,t/2]$. Put $p_\alpha = P\big( W_l^\alpha = 1 \big)$. 
Now for each $l$ we rely on the discrete time contact process $(\xi_n)_{n\ge 0}$ starting at an occupied site in $\Gamma_l$, where the parameter $\theta$ from Lemma \ref{lem:coupling} is to be considered sufficiently large.

From section 8 and 14 in \cite{d2}, it follows that $p_\alpha$ can be as close to one as necessary by taking $R$ sufficiently large, as far as $p$ and $\hat{p}$ are both greater then the critical probability for the dicrete time contact process and $\alpha$ is smaller than $\mathrm{P} (0 \in \xi^{\mathbb{Z}_+}_\infty)$, i.e. the probability that $0$ is occupied under the upper invariant measure for the contact process. Note that $\lim_{p,\hat{p} \ra 1} \mathrm{P} (0 \in \xi^{\mathbb{Z}_+}_\infty) = 1$, see section 14 in \cite{d2}. Morever, since the events $\{W_l^\alpha = 1\}$ are increasing, from the FKG inequality we have that $P(W_l^\alpha = 1|W_k^\alpha = 1) \ge P(W_l^\alpha = 1) = p_\alpha$, for every $1\le l, k \le \big\lceil (Lt + 1)/R \big\rceil$. Therefore from Theorem 0.0 in \cite{lss}, if $p_\alpha > 3/4$ then we have that the $W_l's$ are stochastically dominated from below by iid Bernoulli random variables of parameter $\tilde{p}_\alpha = \tilde{p}_\alpha(q,\theta,R)$ such that $\lim \tilde{p}_\alpha = 1$ as $q \ra 1$, $\theta \ra \infty$ and $R \ra \infty$.

Now from the large deviations for iid Bernoulli random variables, for each $\epsilon > 0$ fixed, outside an event of exponentially small probability (i.e. $e^{-ct}$ for $c > 0$ depending on $\alpha$, $\theta$, $R$ and $\epsilon$), we have that
$$
\sum_{l=1}^{\big\lceil (Lt + 1)/R \big\rceil} W_l^\alpha \ge (\tilde{p}_\alpha - \epsilon) \, \Big\lceil \frac{Lt + 1}{R} \Big\rceil \, ,
$$
which implies that
$$
\frac{\# \big\{ j \in \{0,1,...,Lt-1\}  \, : \ \eta_t (z_j) =1 \big\}}{Lt} \ge \alpha (\tilde{p}_\alpha - \epsilon) \, .
$$
Now, simply choose $\alpha$, $\theta$, $R$ and $\epsilon$ such that $\alpha (\tilde{p}_\alpha - \epsilon) > \rho$ to finish the proof of the inequality in the statement.

For the case $\nu = \delta_y$, one needs only to note that the dependence on $y$ comes from Lemma \ref{lem:navigated}.
\end{proof}

\bigskip 

\subsection{Second Stage.} \label{secondstage}
In this section we define a percolation structure based on the Harris graphical construction of the FA1f and a semi-oriented percolation model (which can also be thought of as a discrete time contact process), similar to and related to processes considered in Section \ref{firststage}.  We will be motivated by trying to understand the 
dual process of our FA1f model $(\eta_t)_{t \geq 0}$.

We fix a constant $K>0$ and consider $t>4K$ also fixed.
Now we renormalize time and discretize space time via (dual) intervals 
$$I(y,i) = \{y\} \times [iK /2, (i+1)K/2] \subset V \times \mathbb{Z}_+ \, .$$ We say that $(y,i)$ (or equivalently $I(y,i)$) is \emph{good} if the following two conditions hold:
\begin{itemize}
\item[(i)] In the Harris scheme the interval $\{y\} \times [t- (i+1)K/2, t-iK/2] $ contains no type-0 decision point.
\item[(ii)] In the Harris scheme the interval $\{y\} \times [t- (i+2)K/2, t- (i+1)K/2] $ contains at least one type-1 decision point and no type-0 decision point.
\end{itemize}

The importance being that if we are given sites $y = y_0$,  $y_1, \, \cdots y_m$ in $V$ with $y_i \sim y_{i-1}$, for every $i=1,...,m$, then if $(y_i,i)$ is good for each $i$ and $\eta_{t-(mK)}(y_m) = 1$, it follows that $\eta_t (y) = 1$. 

\smallskip

In other to control the probability of an interval being good, from now on we consider
\begin{equation}
\label{qK}
q  = 1 + K^{-1} \log \big( 1 - e^{-K/2} \big) \, .
\end{equation}
With this choice we have that 
$$
e^{-(1-q)K} = (1-  e^{-K/2}) \, , \quad e^{-qK} = \frac{e^{-K}}{(1-e^{-K/2})^\frac{1}{K}} 
$$
and the probability of the event $\{I(y, i) \textrm{ is bad}\}$  is equal to
\begin{equation}
\label{probbad}
p_K = 1 - e^{-(1-q)K} (1-  e^{- qK/2}) \, = \, e^{-K/2} \Big( 1 + \frac{e^{-K/2} (1 - e^{-K/2})}{(1-e^{-K/2})^\frac{1}{K}} \Big) \, ,
\end{equation}
which is bounded above by $2e^{-K/2}$.

\smallskip

We now define the semi-oriented percolation model mentioned above. Recall the definition of $\mathcal{Z} = \{{z_0}, z_{1}, z_{2}, \dots\}$ from Section \ref{intro} and fix $y_{0} \in V$. Let $y_{0}, y_{1}, \dots, y_{r} = z_{j}$ be the shortest path from $y_{0}$ to $\mathcal{Z}$. We let $\mathcal{Z}^{ y_{0}}$ be the copy of $\IZ_{+}$
$$y_{0}, y_{1}, \dots, y_{r}, z_{j+1}, z_{j+2}, \dots$$ 
Let us suppose for the moment that $y_0$ and $k \geq 0$ are fixed.
For $w$ in $\mathcal{Z}^{ y_{0}}$ and $l \geq 0 $ consider Bernoulli random variables $J_k(w,l)$ which are equal to one if and only if $I(w,k+l)$ is good. Then the random variables $J_k(w,l)$ are independent of all other $J_k(u,l')$ random variables except $u=w$ and $|l-l'| = 1$.  We derive our one-sided semi-oriented process $\xi^{y_{0},k}$ on $\{0,1\}^{\mathcal{Z}^{ y_{0}}}$ by 
\begin{equation*} 
\xi_{0}^{y_{0},k}(x) = \delta_{y_0}(x) :=
\left\{
\begin{array}{cl}
1 &, \ x = y_{0} \\
0 &, \ \textrm{otherwise} ,
\end{array}
\right.
\end{equation*}
and
$$
\xi_{n}^{y_{0},k}(x) = 1 \mbox{ if and only if } J_k(x,n) =1 \mbox{ and }   \xi_{n-1}^{y_{0},k}(w) = 1 
$$ 
for $w$ a neighbouring site in $\mathcal{Z}^{ y_{0}}$.  (So in particular $\xi_{n}^{y_{0},k}(x) = 1$
is only possible for $n+ d(y_0, x)$ even.)

To motivate this process note that if for some $n\ge 1$ and $w \in \mathcal{Z}^{ y_{0}}$ we have that $\eta_{t-(i+n+2)K/2} (w) = 1$ and $\xi_{n}^{y_{0},i}(w) = 1$, then $\eta_{t-iK/2} (y_0) = 1$.





The processes $\left( \xi_{n}^{y_{0},k} \right)_{n \geq 0}$, $y_0 \in V$, $k \ge 1$, are identically distributed (up to the time where they are defined and relabelling of the sites). So we consider a semi-oriented process  $\left( \xi_{n} \right)_{n \geq 0}$ on $\{0,1\}^{\mathbb{Z}_{+}}$ that evolves as the $\left( \xi_{n}^{y_{0},k} \right)_{n \geq 0}$ and starts at $\xi_{0} = \delta_{0}$. This is the same notation used in Section \ref{firststage}, although the processes are not the same. There is no prejudice since the contact process of Section \ref{firststage} is not used outside that section, moreover the results we state below hold in both cases.

Though the process is defined via site associated random variables, we will regard the semi-oriented percolation process as a 1-dependent bond percolation model on bonds
$$
((x,n),(x-1,n+1)) \, , \, ((x,n),(x+1,n+1)) 
$$
for $n \ge 1$ and $x \ge 1$ with $x+n$ even and 
$$
((0,2n),(1,2n+1)) \, ,
$$
for $n \ge 1$.

\smallskip

Our overall aim is to show that if the semi-oriented process dies out, the die out time has exponentially decaying tail and that if the process survives it must give many occupied sites. Recall that the  processes will be on half lines rather than on $\IZ$, since we are guaranteed half lines but not necessarily copies of $\IZ$ in our graph.

\smallskip

Put $\nu = \inf \{n \geq 0: \xi_{n} \equiv 0\}$. The first result we need is the following:

\medskip

\begin{proposition} \label{death}
There exists $C>0$ so that 
$$P \big( \nu \ge n \, , \, \nu < \infty \big) \le C \, e^{- \frac{K}{4} n}$$
for every $K$ sufficiently large. 
\end{proposition}

\begin{proof}
This result is a direct result of the contour arguments found in \cite{d2}.  We denote by 
$\Gamma$ the collection of $(m,n)$ with $\xi_{n}(m) = 1$ and take
$$
D \ = \ \cup _{(m,n) \in \Gamma} Q_{m,n}
$$
where $Q_{m,n}$ is the unit sided square centred at ${m,n}$ whose edges are at  angle $\pi / 4$ to the axes.   Thus, if $\nu = n$, then $D \subset [- 1/2, n- 1/2] \times  \IZ_+$ but is not contained in 
$[- 1/2, n- 3/2] \times  \IZ_+$.  Let $\delta D$ denote the outer boundary of $D$.  Then $\delta D$
consists of an even number of unit edges at angle $\pi / 4$ to the axes.  These edges form a path which we will regard as starting at $(0, - 1/2)$ and ending there.   If we traverse 
$\delta D$ in a counter clockwise orientation, then each edge with a direction from right to left (whether up or down) logically implies that a given fixed bond is "closed".  By following these edges of $\delta D$
we arrive at a first time (after the initial edge) where the edge touches $\{0\} \times \IZ_+$, we arrive at a contour of an even number of edges which on event $\{\nu = n\}$ will be of length greater than 
or equal to $2n$.  Necessarily this contour must have as many right to left edges as left to right.  
Thus for such a contour of length $2m$, for it to be derived from $\delta D$ requires that a non random collection of $m$ edges be closed.  By the one dependent structure of our model, this entails that at least $m/2$ fixed intervals must be bad.  This and standard contour counting bounds gives the result.

\end{proof}
 
\medskip

Here we simply record some simple but useful properties for the semi-oriented  process $\left( \xi_{n} \right)_{n \geq 0}$ for $K$ sufficiently large. We will consider $(\xi_{n})_{n \geq 0}$ under more general (non zero) initial conditions and for the sake of simplifying the statements we consider as $p$ the probability of a given site being good. For the proofs and more on contact processes/oriented percolation models we suggest \cite{d2} and \cite{d3}.

\medskip

\begin{proposition} \label{death2}
For each $\beta < 1$ there exists $p_\beta  < 1$ so that for $p_K \in [p_\beta , 1]$ and $z \in \IZ_{+}$ if $\xi_{0} = \delta_{z}$ then for every $n > 0$ 
$$
P\big( r_{n} < \beta n + z, \nu > n \big) \leq (p_ \beta)^{n} \, .
$$
where $r_{n} = \sup \{x \in \IZ_{+} : \xi_{n} (x) = 1\}$.
\end{proposition}

\bigskip

The latter proposition can be pushed to the following result.

\medskip

\begin{proposition} For every $0 < R < 1$, there exists $\tilde{p}  < 1$ so that, for every $\vert z \vert \leq R \, t$ and $p_K \in  [\tilde{p} , 1]$, if $\xi_{0} = \delta_{z}$ then
$$P\Big( \{\nu > 2 R \, t\} \cap \big\{\exists \, m \geq R \, t : \ r_{m} < \frac{R \, t}{2} \ \textrm{or} \ \xi_{s}(0) = 0 \ \forall \  s \in [R \, t, \, 2 R \, t] \big\} \Big)$$ 
is bounded above by $\tilde{p}^{\, 2 R \, t}$.
\end{proposition}

\medskip

\begin{corollary}
For every $0 < R < 1$, there exists $\tilde{p}  < 1$ so that, for every $\vert z \vert \leq R \, t/2$ and $p_K \in  [\tilde{p} , 1]$, if $\xi_{0} = \delta_{z}$
$$
P\Big(\exists \, n \in (2 R \, t,t) \textrm{ with } \sum_{0 \leq x\leq \frac{R \, t}{2}} \xi_n (x) < \frac{4 R \, t}{20}, \nu \geq 2 R \, t \Big)
$$
is bounded above by
$$t \, P\Big( \sum_{0 \leq x \leq \frac{R \, t}{2}} \hat{\xi}(x) < \frac{4 R \, t}{20} \Big)+ \tilde{p}^{\, 2 R \, t}$$
where $\hat{\xi}$ is a configuration in non trivial equilibrium.
\end{corollary}

\bigskip

We now relate these results to our discrete time process $ \left( \xi_{n}^{y_{0},k} \right)_{n \geq 0}$. We will be interested in two semi-oriented processes. The original process on $\mathcal{Z}^{y_{0}}$ and a related ``subordinate'' process on $\mathcal{Z}$ itself.

\noindent Recall that $y_{0} \in V$ and $k \ge 1$ are fixed. We first note that if $\nu$, the death time for $ \xi_{n}^{y_{0},k}$, is greater than $R \, t$ then outside of probability $e^{-ct}$ we have (for $d(y_0,\mathcal{Z})< \frac{R \, t}{2})$ that $\xi_{n}^{y_{0},k}$ is not empty on $\mathcal{Z} \cap \mathcal{Z}^{y_{0}} \hspace{0.3cm} \forall n \geq R \, t$. We now consider (following \cite{d1}) the stopping times $\nu_{0}, \nu_{1}, \dots$ defined as follows $\nu_{0} = R \, t$ at this time pick a site $y = y_1$ in $\mathcal{Z}$ for which $\xi_{\nu_{0}} (y) = 1$. Let $\nu_{1}$ be the (possibly infinite) time when the semi-oriented process beginning at $\nu_{0}$ with only $y$ occupied on $\mathcal{Z}$ expires. Given $\nu_{i-1}$ let $y$ be replaced by a new site $y_{i}$ in $\mathcal{Z}$ so that $\xi_{\nu_{i}-} (y_{i}) = 1$ and let $\nu_{i}$ be the (possibly infinite) time that the discrete time semi-oriented process in $\mathcal{Z}$ starting at $\nu_{i-1}$ dies. 
The following is a simple consequence of Propositions \ref{death} and \ref{death2}.

\vspace{0.3cm}
\begin{lemma}
For every $0 < R < 1$, there exists $\tilde{p} \in (0,1)$ and $c = c(R,\tilde{p}) > 0$ so that for $d(y_0,\mathcal{Z}) \leq \frac{R \, t}{2}$ and $p_K > \tilde{p}$
$$
P \big( \{\nu \geq R \, t \} \cap E \big) \leq e^{-ct} \, ,
$$
where
$$
E=
\{\textrm{For some choice of } y_1, y_2,...  \textrm{ there is no } i < 2 R \, t \textrm{ with } \nu_{i} = \infty\} \, .
$$
\end{lemma}

\vspace{0.3cm}
\noindent From this we immediately obtain

\begin{proposition} \label{prop:percolation}
For every $0 < R < 1$, there exists $\tilde{p}  < 1$ and $c= c(R,\tilde{p})>0$ so that, for every $d(y_0,z_0) \leq R \, t/4$, $\vert k \vert \leq R \, t / 2$ and $p_K \in  [\tilde{p} , 1]$
$$P \Big( \sum_{j = 0}^{\frac{R \, t}{2}} \xi^{y, k}_{R \, t -k} (z_j) < \frac{R}{5} \, t, \nu \geq R \, t \Big) \leq e^{-c t}.$$
\end{proposition}

\bigskip

\subsection{Third Stage.} \label{thirdstage}
Recall the definition of dual paths and $\mathcal{D}(x,t,\tau)$ from Section \ref{intro}. Here $x \in V$ is a fixed site which is at (graphical) distance $R \, t$ from our ``origin" $z_0$.  We are interested in paths in $\mathcal{D}(x,t,(1-\sigma) t)$. We say a dual path $X \in \mathcal{D}(x,t,(1-\sigma) t)$ \emph{encounters a good percolating interval} $I(y,i)$ if for some $s \in [t-iK/2, t-(i+1)K/2]$, 
$X(t-s)= y$.
 
\medskip

The objective of this section is to show the following result:

\medskip

\begin{proposition} \label{hitperc}
Let $G = (V,E)$ be a countable connected graph of bounded degree satisfying the $(\vartheta,\tilde{\vartheta},\varepsilon)$-growth condition. Let $0<R<1$ and $K>0$ be fixed as in the previous section.  There exists $K_0$ and $\sigma_0 < 1/4$ so that for $K > K_0$, $\sigma < \sigma_0$ and all $t $ large if $|x| \leq R \, t/4 $ fixed, the probability that there exists a dual path in in $\mathcal{D}(x,t,(1-\sigma) t)$ which does not encounter a
$K$ normalized ''dual" contact process that survives until time $t/4$ and touchs at least $R/5$ sites among $\{z_{0},...,z_{Rt/2}\}$ at that time is less than $C e^{-ct}$ for some $c=c(K,\vartheta,\tilde{\vartheta},\varepsilon,\sigma)>0$ and $C=C(K,\vartheta,\tilde{\vartheta},\varepsilon,\sigma)>0$.
\end{proposition}

In analyzing dual paths we will use various codings (or discrete representations for these objects.  We begin with a basic coding.
A dual path can be coded (in 1-1 fashion) by a sequence 
$y = y_0, y_1 \cdots y_m $ where $\forall i ,$ $y_i $ and $y_{i-1} $ are either equal or nearest neighbours and so that if we define times $t_i $ recursively by $t_0 = 0$ and for $i > 0$,
$$
t_i \ = \ \inf \{s > t_{i-1}:  ( y_{i-1},t-s ) \mbox{ is a decision point}\},
$$
then $X(s) \ = \ y_i $ on $[t_i, t_{i+1} )$ and $t_{m+1} > t $.  The ``value" of $X$, $m$, is denoted by $|X|$.

\medskip

\begin{lemma} \label{lemcount1} 
For every sufficiently large $N$, we have that
$$P \big( \exists \ X \in \mathcal{D}(y_0,t,s) \textrm{ with } \vert X \vert > N (t-s) \big) \leq e^{-t} \, , $$
for every $0 \le s < t$. 
\end{lemma}

\begin{proof}
We will consider the case $s=0$, it should be clear that the proof holds for $0<s<t$. The statement of the lemma is that we cannot find $y_0, y_1 \cdots y_{N t}$ such that for all $i$, $y_i$ and $y_{i-1}$ are either equal or nearest neighbours and (with the above definition) $\sum_{i=1} ^ {N t} (t_i - t_{i-1}) \ \leq \ t $.  Now there are (at most) $(\kappa +1) ^{N t}$ (recall that $\kappa$ is the degree of the graph) such sequences and the probability that for any such fixed sequence
has $\sum_{i=1} ^ {N t} (t_i - t_{i-1}) \ \leq \ t$ is equal to the probability that $\sum_{i=1} ^ {N t} e_i \ \leq \ t$ for i.i.d. standard exponential random variables $e_i$. So by standard Tchebychev bounds the probability in the statement is bounded above by
$$
\frac{ \big( (\kappa+1) E(e^ {-(\kappa + 1) \, e_1}) \big)^{N t} }{e^{- (\kappa + 1) \, t } }  \  =  \ 
 \Big( \big( \frac{k+1}{k+2} \big)^{N} e^{(\kappa+1)}  \Big)^t  \leq \  e^{-t}
$$
for $N$ large and all $t$ positive.
\end{proof}

We now consider  a coding of a dual path $X$ which is ``compatible" with the discretization imposed by the renormalization procedure of Section \ref{secondstage}.  Given the coding $y = y_0, y_1 , \cdots y_m$  (given Lemma \ref{lemcount1} we may and shall assume that $m < N t$), we define a skeleton of it $(v_1, v_2 \cdots v_{t/(2K)})$ to be such that for all $i$ in time interval
$[(i-1)K/2, iK/2]$, the path $X$ begins at a site $z^i_a $ and ends at site $z^i_b$ which are linked by a path of $v_i$ sites each visited by $X$ in this interval. 
Thus, for every dual path $X$ which is coded as $y = y_0, y_1 , \cdots y_m$ we have a renormalized coding (not uniquely defined)
$$
(y_{0}\cdots  y_{v_{1}}) \, , \ (y_{v_{1}}\cdots  y_{v_{1}+ v_{2}}) \, , \, ... \, , \ 
   (y_{v_{1}+ ... + v_{(t/2K)-1}}\cdots  y_{v_{1}+ ... + v_{t/2K}}) \, .
$$
For instance e.g. $y_{v_{1}+ v_{2}}\cdots  y_{v_{1}+ v_{2} + v_{3}}$ represents a $v_3$ path of visited sites from the first visited site to the last on the third time interval. We denote by $\{y_i\} \{v_j\}$ a renormalized coding, i.e. a pair where $\{y_i\}$ is a coding and $\{v_j\}$ is its associated skeleton.

\medskip

\begin{lemma} \label{lembound1} 
For $0 < \epsilon < 1$ and a fixed renormalized coding corresponding to a dual path of size less than $N t$, the probability that more than $\epsilon t /K$ of the intervals visited are bad is less than $C e^{- \frac{\epsilon}{4} t}$ for all $K$ sufficiently large, where $C=C(K,N) > 0$ does not depend on the chosen path.
\end{lemma}

\begin{proof}
Let us simply remark that the intervals at a fixed time level are independent, while given the information on the status up to (dual) time $(i+1)K/2$, the status of $I(y_j, i)$ are conditionally independent for $y_j \, \in \, y_{v_{1}+ v_{2}\cdots v_i }\cdots  y_{v_{1}+ v_{2} + \cdots y_{v_{i+1}}}$ and by \eqref{probbad} $P(I(y_j, i)$ is good $| \mathcal{F} _i ) \ \geq \ 1 - 2 e^{-\frac{K}{2}}$ if either $I(y_j,i-1)$ is not identified or is good.  Thus we easily see our probability is bounded by 
the probability that a binomial with parameters $Nt$ and $2 e^{-\frac{K}{2}}$ has value greater than $\epsilon t / K$.  This binomial probability is bounded above by
$$
2 \sum_{j\ge \epsilon t/K} \binom{Nt}{j} e^{-K j/2} 
$$
which, by a straightforward computation using Stirling formula, is bounded above by some term that grows polynomially in $t$ times 
$$
\exp\Big\{\epsilon \, t \Big( \frac{\log(N)+1)}{K} - \frac{1}{2} \Big) \Big\} \, .
$$
To finish the proof we just need to take $K > 4 (\log(N) + 1)$ and adjust the constants. 
\end{proof}

\medskip

Since we are interested in the event that some dual path never encounters a good interval which percolates for time $t/4$.  Were this to happen then some renormalized coding would never  encounter a good interval which percolates.  Then every interval encountered would either be bad (which by Lemma \ref{lembound1} for large enough $q$ would only be a small proportion) or must have a finite percolation lifetime.  Thus (unless the bound of Lemma \ref{lembound1} is violated) we must be able to find 
a collection of levels $i_1, i_2, \cdots i_f $ and associated to each level $i_j$ a $w_j \ \in \ y_{v_{1}+ v_{2}\cdots v_{i_j} }\cdots  y_{v_{1}+ v_{2} + \cdots v_{i_j+1} -1}$ so that $I(w_j, i_j)$ is good but its percolation lasts for time $\ell^{w_{j}} $ and so that the size of $|\cup _j [i_j, i_j+ \ell^{w_{j}} ] \vert \geq \frac{t}{4K} - \frac{\epsilon t}{K}$.

Choosing $\epsilon$ sufficiently small, by Vitali Covering Lemma we can find 
$i_{1' }, i_{2' } , \cdots i_{f' }$ so that
\begin{itemize}
\item[(i)]$ \hspace{0.3cm} \forall j' \not= j'' \hspace{0.3cm} [i_{j'}, i_{j'} +\ell^{w_{j'}}] \cap [i_{j''}, i_{j''} +\ell^{w_{j''}}] = \emptyset$
\item[(ii)] $\vert \cup_{j'=1}^{f'} [i_{j'}, i_{j'}+ \ell^{w_{j'}}] \vert = \sum_{j=1}^{f'} \ell^{w_{{j'}}}  \geq \frac{t}{15K}.$
\end{itemize}

\medskip

So we want to count the number of Vitali Coverings associated to a dual path of length at most $Nt$. For this we do not need to count all associated renormalized coding $\{y_i\} \{v_j\}$ since we only need to pick one good interval $I(w_j, i_j)$ by time level, with $$w_j \ \in \ \{ y_{v_{1}+ v_{2}\cdots v_{i_j} } \, , \cdots , \,  y_{v_{1}+ v_{2} + \cdots v_{i_j+1} -1} \}.$$ We call $\{w_j,v_j\}$ a Vitali coding for the dual path $X$. There are multiple Vitali codings for a given $X$ but the next result shows that there are at most $K^{\frac{2ct}{K}}$ such codings.

\medskip

\begin{lemma} \label{lemcount2} Let $G = (V,E)$ be a countable connected graph of bounded degree satisfying the $(\vartheta,\tilde{\vartheta},\varepsilon)$-growth condition. It follows that:
\begin{itemize}
\item[(i)] There are at most $\sum^{N \, t}_{L=0} \binom {L + t/(2K)}{t/(2K)} \leq K^{\frac{ct}{K}}$ choices of skeleton corresponding to dual paths of size less than $N t$ for some $c > 0$ not depending on $t$ and $K > 0$.\\
\item[(ii)] Given $\tilde{v}= (v_{1}, v_{2}, \cdots, v_{\frac{t}{4K}} )$ there are at most $e^{\frac{ct}{K^\varepsilon}}$ choices of corresponding Vitali codings for some $c > 0$ not depending on $t$ and $K >0$.
\end{itemize}
\end{lemma}

\begin{proof}
We note first that $\sum _ i v_i \ \leq \ N t$ and if $L$ is the sum, the number of skeletons is exactly  $\binom {L + t/(2K)}{t/(2K)}$. By summing over $L$ we can get an upper bound of $\binom {Nt + t/(2K)+1}{t/(2K)+1}$ and inequality (i) follows by an application of Stirling formula.  

Part (ii) follows from the standard path counting. Here we use the $(\vartheta,\tilde{\vartheta},\varepsilon)$-growth condition which gives a number of corresponding codings of at most 
\begin{eqnarray*}
\prod_{j=1}^{\frac{t}{4K}} \vartheta \exp\{\tilde{\vartheta} \, v_j^{1-\varepsilon}\} & = & \vartheta^{\frac{t}{4K}} \, \exp\Big\{\tilde{\vartheta} \sum_{j=1}^{\frac{t}{4K}} v_j^{1-\varepsilon} \Big\} \\
& \le & \vartheta^{\frac{t}{4K}} \, \exp\Big\{\frac{\tilde{\vartheta} \, t}{4K} \Big( \frac{4K}{t} \sum_{j=1}^{\frac{t}{4K}} v_j \Big)^{1-\varepsilon} \Big\} \, .
\end{eqnarray*}
Now use the fact that $\sum_{j=1}^{\frac{t}{4K}} v_j \le Nt$ to get the bound in the statement.
\end{proof}



\medskip

\begin{remark}
We only use the $(\vartheta,\tilde{\vartheta},\varepsilon)$-growth condition in the proof of Lemma \ref{lemcount2}.
\end{remark}

\vspace{0.3cm}

\begin{lemma} \label{vitalliprob}
For a fixed Vitali coding $\{w_j,v_j\}$ as above the probability of $i_{1'}, i_{2'}, \cdots, i_{f'}$ giving such intervals is at most $e^{- \frac{t}{120}}$ for all sufficiently large $K$.
\end{lemma}

\smallskip
\begin{proof}
Recall Proposition \ref{death} and note that $\ell^{w_{{j'}}}$ have the same distribution as $\nu$. Indeed $\ell^{w_{j'}}$ is the time of extinction of the renormalized contact process starting at the good interval $I(w_{j'}, i_{j'})$. Futhermore we also have independence of $\ell^{w_{{1'}}}$, ... , $\ell^{w_{{f'}}}$ since our assumption is that the initial renormalized intervals for each interval $[i_{j'}, i_{j'} + \ell^{w_{j'}}]$ is good and these intervals have finite length. 

Therefore, considering the possible ways of choosing the lengths $\ell^{w_{i,j'}}$, by Proposition \ref{death} the probability in the statement is 
\begin{eqnarray}
\lefteqn{ \!\!\!\!\!\!\!\!\!\!\!\!\!\!\!\!\!\!\!\!
\sum_{\frac{t}{15K} \le n_1 + ... +n_{f'} \le \frac{t}{4K}} \prod_{j=1}^{f'} P \Big( \ell^{w_{{j'}}} = n_j \Big) \, \le } \\
& &  \le \, \sum_{\frac{t}{15K} \le n_1 + ... +n_{f'} \le \frac{t}{4K}} \prod_{j=1}^{f'} e^{- n_j \frac{K}{4}} \\
& & = \, \sum_{\frac{t}{15K} \le n_1 + ... +n_{f'} \le \frac{t}{4K}}  e^{- \sum_{j=1}^{f'} n_j \frac{K}{4}} 
\, \le \, \frac{t}{4K} \, 2^{\frac{t}{4K}} \, e^{- \frac{t}{15K} \frac{K}{4}} \, .
\end{eqnarray}
Now choose $K$ sufficiently large and the rightmost term in the previous inequality is bounded above by $e^{- \frac{t}{120}}$.
\end{proof}

\bigskip

\noindent {\it Proof of Proposition \ref{hitperc}.}
From Lemmas \ref{lembound1} and \ref{vitalliprob}, the probability that some Vitali coding of length smaller than $N t$ fails to touch a good interval is bounded above by
$$
e^{-\frac{\epsilon}{4} t } \, + \,  e^{- \frac{t}{120}} .
$$

Now fix $\sigma < \frac{R}{4N} \wedge \frac{1}{4}$. For a dual path $X$ in $\mathcal{D}(x,t,(1-\sigma) t)$ of length smaller than $\sigma N t \le  Rt/4$ we have that if $X$ touchs a good interval containing $(w,s) \in V \times [(1-\sigma) t,t]$ then $d(w,z_0) \le Rt/2$. By Proposition \ref{prop:percolation}, outside an event of exponentially small probability, a good interval touched by $X$ percolates until time $t/4$ and its percolation cluster touchs at least $R/5$ sites among $\{z_0,...,z_{Rt/2}\}$ at that time.

Therefore from Proposition \ref{prop:percolation} and Lemmas \ref{lembound1} and \ref{vitalliprob} , the probability that some Vitali coding of length smaller than $\sigma N t$ fails to touch a good interval that percolates until time $t/4$ and its percolation cluster touchs at least $R/5$ sites among $\{z_0,...,z_{Rt/2}\}$ at that time is bounded by
$$
e^{-\frac{\epsilon}{4} t } \, + \,  e^{- \frac{t}{120}} \, + \,  e^{- ct} \, .
$$

Hence (using Lemma \ref{lemcount1} and Lemma \ref{lemcount2}) the probability that there exists a dual path not meeting a percolating interval is bounded by 
$$
e^{-t} + K^{\frac{ct}{K}} \, e^{\frac{ct}{K^\varepsilon}} \big( e^{-\frac{\epsilon}{4} t } \, + \,  e^{- \frac{t}{120}} \, + \,  e^{- ct}  \big)
$$ 
and the result follows. $\square$

\bigskip

\subsection{Proof of Proposition \ref{prop:main2}.\\
\\}

To prove Proposition \ref{prop:main2} we only need to obtain the inequality in the statement for $t$ sufficiently large (depending on $x$ and $y$), $c$ and $C$ not depending on $x$ and $y$. Then we can increase $C$ according to the choices of $x$ and $y$ to obtain the statement as it is presented.

Fix sites $x,y \in V$ as in the statement of Proposition \ref{prop:main2}. Consider $t$ sufficiently large such that $d(x,z_0) \vee d(y,z_0) \le Rt/4$ for some fixed suitable $0< R < \frac{2q-1}{2}$. Now fix $N$ as in Lemma \ref{lemcount1} and, as in the proof of Proposition \ref{hitperc}, choose $\sigma < \frac{R}{4N} \wedge \frac{1}{4}$. 

By Proposition \ref{hitperc}, we can fix $q$ sufficiently close to one (or $K=K(q)$ sufficiently large) so that outside an event of probability $C e^{-ct}$ for $c=c(q,\vartheta,\tilde{\vartheta},\varepsilon)>0$ and $C=C(q,\vartheta,\tilde{\vartheta},\varepsilon)>0$, every path $X$ in $\mathcal{D}(x,t,(1-\sigma) t)$ 
touchs at some point $(w,s) \in V \times [(1-\sigma) t,t]$ a $K$ normalized ''dual" contact process that survives until time $t/4$ and touchs at least $R/5$ sites among $\{z_0,...,z_{Rt/2}\}$ at that time. 

By Proposition \ref{lemma:densidade} at least $9R/20$ sites among the same $\{z_0,...,z_{Rt/2}\}$ are occupied for both processes $\eta$ and $\tilde{\eta}$ at time $t/4$ with probability $1-Ce^{-ct}$ for $c=c(q,\sigma)$ and $C=C(q,\sigma)$. Therefore, outside an event of probability $Ce^{-ct}$ for some $c=c(q,\vartheta,\tilde{\vartheta},\varepsilon,\sigma)>0$ and $C=C(q,\vartheta,\tilde{\vartheta},\varepsilon,\sigma)>0$, for every path $X$ in $\mathcal{D}(x,t,(1-\sigma) t)$ with $(w,s) \in V \times [(1-\sigma) t,t]$ as above there exists $z_j$ such that $\eta_{t/4}(z_j) = \tilde{\eta}_{t/4}(z_j) = 1$ and this one is carried by a navigating path to $w$ at time $s$, i.e, we also have $\eta_s(w) = \tilde{\eta}_s(w) = 1$, thus $X$ is $t$-activated. By an appropriate choice of the constants, we obtain Proposition \ref{prop:main2}.

\end{document}